\documentclass[12pt]{article}

\usepackage{amsthm,amssymb,amsmath,bbm,color,epsfig,psfrag}

\newtheorem{superclass}{superclass}
\newtheorem{definition}[superclass]{Definition}

\newtheorem{theorem}[superclass]{Theorem}
\newtheorem{proposition}[superclass]{Proposition}

\newcommand{\mc}{\mathcal}
\newcommand{\R}{\mathbbm{R}}
\newcommand{\N}{\mathbbm{N}}
\newcommand{\B}{\mathbb{B}}
\newcommand{\borel}{\mc{B}}
\renewcommand{\S}{\mathbb{S}}
\newcommand{\vertiii}[1]{{\left\vert\kern-0.25ex\left\vert\kern-0.25ex\left\vert #1 
    \right\vert\kern-0.25ex\right\vert\kern-0.25ex\right\vert}}

\DeclareMathOperator{\dist}{dist}

\DeclareMathOperator{\proj}{proj}

\begin{document}

\title{A reinterpretation of set differential equations as differential equations in a Banach space}
\author{Martin Rasmussen, Janosch Rieger\footnote{Supported by a Marie Curie Fellowship of the European Union.} and Kevin Webster\\
\small{Department of Mathematics, Imperial College London}}
\date{\today}
\maketitle

\begin{abstract}
Set differential equations are usually formulated in terms of the Hukuhara differential,
which implies heavy restrictions for the nature of a solution.
We propose to reformulate set differential equations as ordinary differential 
equations in a Banach space by identifying the convex and compact subsets of $\R^d$ with 
their support functions.
Using this representation, we demonstrate how existence and uniqueness results can be 
applied to set differential equations.
We provide a simple example, which can be treated in support function representation,
but not in the Hukuhara setting.
\end{abstract}


\section{Introduction}

A set differential equation is an equation of the form
\begin{equation} \label{Hukuhara:SDE}
D_H A(t)=f(t,A(t)),\quad A(0)=A_0,
\end{equation}
where $t\mapsto A(t)$ is a curve in the space $\mc{K}_c(\R^d)$ of nonempty convex and compact subsets of $\R^d$,
the right-hand side is a mapping
\[f:[0,T]\times\mc{K}_c(\R^d)\to\mc{K}_c(\R^d),\]
and $D_H A(t)$ is the so-called Hukuhara differential of the curve at $t\in(0,T)$.
Set differential equations have been investigated in a considerable number of papers.
For an overview of the literature we refer to \cite{Lakshmikantham}.
The usage of the Hukuhara differential in \eqref{Hukuhara:SDE} implies heavy restrictions for the nature of a solution, 
which can, e.g., only grow in diameter, but not shrink, see \cite[Proposition 1.6.1]{Lakshmikantham}.

Recently, there have been attempts to modify the underlying Hukuhara difference with the 
aim to allow for a more flexible behavior of solution curves, see \cite{Malinowski:12a,Malinowski:12}
and the references therein. 
The resulting differential is called the second type Hukuhara differential.
In this setting, solution curves of \eqref{Hukuhara:SDE} can shrink, but not grow.
There exist, however, curves in $\mc{K}_c(\R^d)$ with $d\ge2$, which expand in some space directions and contract in others simultaneously.
Both Hukuhara-based approaches fail to capture this behavior.

\medskip

The Hukuhara differential is not the only approach to handle set evolutions.
In particular, we would like to mention an abstract framework named \emph{Mutational Analysis},
which has been presented in \cite{Aubin:99} and further developed in \cite{Lorenz:10}.
It generalizes evolution equations from vector spaces to metric spaces and can not only 
handle evolutions in $\mc{K}_c(\R^d)$, but also in spaces of more general sets such as
the compact subsets of $\R^d$.

\medskip

The aim of the present paper is to show that a large family of evolutions in $\mc{K}_c(\R^d)$, 
containing the problems investigated in \cite{Lakshmikantham} and \cite{Malinowski:12a,Malinowski:12},
can be written and treated as ordinary differential equations in a Banach space with the usual Frechet derivative in time.
We do not apply the apparatus from \cite{Aubin:99} and \cite{Lorenz:10}, but obtain very satisfactory results
by exploiting the intrinsic features of of $\mc{K}_c(\R^d)$.

Identifying convex sets with their support functions yields an embedding of the space $\mc{K}_c(\R^d)$ into the 
Banach space $C(\S^{d-1})$ of continuous real-valued functions on the sphere, see \cite{Hormander:55}. 
As it is well-known that any Hukuhara differentiable curve is Frechet differentiable in support function representation,
see \cite[Lemma 4.1]{Banks:Jacobs:70}, it seems natural to consider set differential equations in support function representation
\begin{equation} \label{support:ODE}
\tfrac{d}{dt}\sigma_{A(t)} = f(t,\sigma_{A(t)}),\quad \sigma_{A(0)}=\sigma_{A_0},
\end{equation}
where $t\mapsto A(t)$ is a curve in $\mc{K}_c(\R^d)$, $t\mapsto\sigma_{A(t)}$ is a curve in $C(\S^{d-1})$,
and $\tfrac{d}{dt}\sigma_{A(t)}$ is the Frechet differential of the curve at $t\in(0,T)$.

There are some technical difficulties when standard results on ordinary differential equations 
are applied to equations of type \eqref{support:ODE}.
As we have to guarantee that solutions stay in the manifold $\Sigma\subset C(\S^{d-1})$ of all support functions 
associated with sets from $\mc{K}_c(\R^d)$, we have to understand the structure of the tangent 
cone $T_\Sigma(\sigma)$ to $\Sigma$ at any $\sigma\in\Sigma$.
To transfer existence and uniqueness theorems for ordinary differential equations in Banach spaces 
with non-Lipschitz right-hand side to \eqref{support:ODE}, we need compactness properties 
of $\Sigma$ and a characterization of the semi-inner product on $(C(\S^{d-1}),\|\cdot\|_\infty)$.
Some of these preliminary results can be taken from the literature, others are developed
in the present paper.
In particular, we give a geometric interpretation of the one-sided Lipschitz condition in $\mc{K}_c(\R^d)$,
which is a surprisingly mild condition on the behavior of $f$.

\medskip

The organization of the paper is as follows. 
In Section \ref{sec:prel}, we collect basic definitions and the preliminary results mentioned above,
which we use in Section \ref{sec:ex:uni} to transfer standard existence and uniqueness results to \eqref{support:ODE}.
In Section \ref{sec:second}, we briefly show that second-type Hukuhara differentiable curves are a special case of \eqref{support:ODE}.
The example discussed in Section \ref{counterexample} illustrates that both Hukuhara approaches fail to 
capture very simple dynamics in $\mc{K}_c(\R^2)$, while the support function calculus is applicable
and yields reasonable solutions.

\section{Preliminaries} \label{sec:prel}

After introducing basic notation in Section \ref{sub:definitions}, we will collect some known results 
about support functions and tangent cones in Sections \ref{sub:support} and \ref{sub:tangent}. 
Section \ref{sub:duality} investigates duality concepts, which are ingredients for standard results 
on ordinary differential equations in Banach spaces, in the particular case of set differential equations.

\subsection{Basic definitions} \label{sub:definitions}

Let $\R_0^+$ be the set of all nonnegative real numbers.
Throughout this paper, $\S^{d-1}\subset\R^d$ will denote the sphere w.r.t.\ the Euclidean norm 
$\|\cdot\|:\R^d\to\R_0^+$, and the modulus will be denoted $|\cdot|:\R\to\R_0^+$.
Let $C(\S^{d-1})$ be the space of continuous real-valued functions on $\S^{d-1}$ equipped with 
the maximum norm $\|\cdot\|_\infty:C(\S^{d-1})\to\R_0^+$.
If $(X,\|\cdot\|_X)$ is a normed space, $x\in X$ and $r>0$, then 
\[\B_r(x):=\{x'\in X:\|x'-x\|\le r\}\]
is the closed ball of radius $r$ centered at $x$.

The nonempty compact subsets of $\R^d$ will be denoted $\mc{K}(\R^d)$,
and $\mc{K}_c(\R^d)$ will stand for the nonempty convex and compact subsets of $\R^d$.
For any $\lambda\in\R$ and $A,B\in\mc{K}(\R^d)$, let
\[A+B:=\{a+b:a\in A,\ b\in B\}\quad\text{and}\quad\lambda A:=\{\lambda a:a\in A\}\]
denote Minkowski addition and multiplication.
For any $A,B\in\mc{K}_c(\R^d)$, let
\begin{align*}
\dist(A,B)&:=\sup_{a\in A}\inf_{b\in B}\|a-b\|,\\
\dist_H(A,B)&:=\max\{\dist(A,B),\dist(B,A)\}
\end{align*}
denote the one-sided and the symmetric Hausdorff distance.
For $a,b\in\R^d$, we write $\dist(a,B)$ and $\dist(A,b)$
instead of $\dist(\{a\},B)$ and $\dist(A,\{b\})$.
The projection of a point $a\in\R^d$ to a set $B\in\mc{K}(\R^d)$ is the nonempty set
\[\proj_B(a):=\{b\in B:\|a-b\|=\dist(a,B)\}.\]
When $B\in\mc{K}_c(\R^d)$, then $a\mapsto\proj_B(a)$ is a single-valued mapping, 
see \cite[Lemma 7.3]{Clarke:13}, and it follows from \cite[Proposition 7.4]{Clarke:13} 
that this mapping is $1$-Lipschitz.

We associate convex and compact subsets $A\in\mc{K}_c(\R^d)$ with their 
support functions
\[\sigma_A:\S^{d-1}\to\R,\quad \sigma_A(p):=\sup_{a\in A}\langle p,a\rangle.\]
Sometimes, it is useful to consider their positive homogeneous extensions
\[\bar{\sigma}_A:\R^d\to\R,\quad \sigma_A(p):=\sup_{a\in A}\langle p,a\rangle,\]
which obviously coincide with $\sigma_A(\cdot)$ on $\S^{d-1}$.
We define
\[\Sigma(\R^d):=\{\sigma_A:A\in\mc{K}_c(\R^d)\}\]
to be the set of all support functions of convex and compact subsets of $\R^d$,
and we set
\[\hat\Sigma(\R^d):=\Sigma(\R^d)-\Sigma(\R^d)=\{\sigma_A-\sigma_B:A,B\in\mc{K}_c(\R^d)\}.\]

\subsection{Elementary facts about support functions} \label{sub:support}

The following proposition is Corollary 13.2.2 from \cite{Rockafellar:70}.
\begin{proposition} \label{pos:hom}
A bounded function $\sigma:\S^{d-1}\to\R$ is a support function of some $A\in\mc{K}_c(\R^d)$ if 
and only if its positive homogeneous extension $\bar\sigma:\R^d\to\R$ is convex.
\end{proposition}
Recall that every convex function $\bar\sigma:\R^d\to\R$ is continuous (see \cite[Theorem 10.1]{Rockafellar:70}).
We may therefore interpret the set $\Sigma(\R^d)$ of all support functions as a subset of $ C(\S^{d-1})$.

The following facts are well-known (see \cite{Hormander:55,Hu:Papageorgiou:97}).
\begin{proposition} \label{elementary:support}
If $A,B\in\mc{K}_c(\R^d)$ and $\lambda\ge 0$, then
\begin{itemize}
\item [a)] $\sigma_{A+B}=\sigma_A+\sigma_B$ and $\sigma_{\lambda A}=\lambda\sigma_A$,
\item [b)] $\dist(A,B)=\max_{p\in\B_1(0)}\big(\bar{\sigma}_A(p)-\bar\sigma_B(p)\big)$,
\item [c)] $\dist_H(A,B)=\max_{p\in\S^{d-1}}|\sigma_A(p)-\sigma_B(p)|$.
\end{itemize}
In particular, $\Sigma(\R^d)$ is a convex subcone of $C(\S^{d-1})$.
\end{proposition}

The cone $\Sigma(\R^d)$ is locally compact.
\begin{proposition} \label{locally:compact}
The cone $\Sigma(\R^d)$ is closed as a subset of $C(\S^{d-1})$, and for any $\sigma\in\Sigma(\R^d)$ and $r>0$, 
the intersection $\Sigma(\R^d)\cap\B_r(\sigma)\subset C(\S^{d-1})$ is compact w.r.t.\ the maximum norm.
\end{proposition}
\begin{proof}
Let $\sigma\in C(\S^{d-1})$, and let $(\sigma_n)_{n\in\N}\subset\Sigma(\R^d)$ be a sequence of support functions with 
$\|\sigma_n-\sigma\|_\infty\to 0$ as $n\to\infty$.
By Proposition \ref{pos:hom}, the extensions $\bar{\sigma}_n$ are convex.
Hence, we have for any $\lambda\in[0,1]$ and $x,y\in\R^d$ that
\begin{align*}
\bar{\sigma}(\lambda x+(1-\lambda)y) \leftarrow& \bar{\sigma}_n(\lambda x+(1-\lambda)y)\\ 
\le& \lambda\bar{\sigma}_n(x)+(1-\lambda)\bar{\sigma}(y) \to \lambda\bar{\sigma}(x)+(1-\lambda)\bar{\sigma}(y)
\end{align*}
as $n\to\infty$, so that $\bar\sigma$ is convex.
Therefore, Prosition \ref{pos:hom} implies that $\sigma\in\Sigma(\R^d)$.

By Blaschke's selection theorem, see \cite[Chapter 4]{Eggleston:58}, the set $\Sigma(\R^d)\cap\B_{\|\sigma\|_\infty+r}(0)$
is compact. As $\Sigma(\R^d)\cap\B_r(\sigma)$ is the intersection of two closed sets, it is a closed subset 
of the compact set $\Sigma(\R^d)\cap\B_{\|\sigma\|_\infty+r}(0)$, and hence compact.
\end{proof}

\subsection{Tangent cones} \label{sub:tangent}

We are interested in $C(\S^{d-1})$-valued solutions of differential equations that do not leave $\Sigma(\R^d)$.
The concept of tangency is central for existence theorems under state constraints.
\begin{definition}
Let $X$ be a normed space, $K\subset X$ a set and $x\in\overline{K}$. 
Then the tangent cone to $K$ at $x$ is given by
\[T_K(x):=\{v\in X: \liminf_{h\searrow 0}h^{-1}\dist(x+hv,K)=0\}.\]
\end{definition}

The following proposition is Lemma 4.2.5 in \cite{Aubin:Frankowska:90}.
It will later be used to characterize tangency to the convex cone $\Sigma(\R^d)$.
\begin{proposition} \label{tangent:to:cone}
If $X$ is a normed space and $K\subset X$ is a convex cone, then $T_K(x)=\overline{K+\R x}$ for all $x\in K$.
\end{proposition}

\subsection{The semi-inner product for support functions} \label{sub:duality}

In Section \ref{sec:ex:uni}, we will apply a uniqueness theorem for ordinary differential equations 
in Banach spaces to set differential equations in support function representation.
Its main ingredient is a one-sided Lipschitz condition, which is given in terms of a so-called semi-inner product.
Therefore, we investigate in the present paragraph how this product acts on $\hat\Sigma(\R^d)\subset C(\S^{d-1})$
and what this action means for the corresponding elements of $\mc{K}_c(\R^d)$.

\begin{definition} \label{duality:map}
For any Banach space $X$ with dual space $X^*$, the duality map $J:X\rightrightarrows X^*$ 
is given by
\[J(x)=\{x^*\in X^*: x^*(x)=\|x\|_X^2=\|x^*\|_{X^*}^2\}.\]
The mapping $\langle\cdot,\cdot\rangle_- : X\times X\rightarrow\mathbb{R}$ defined by
\[\langle x, y \rangle_- = \inf\{y^*(x)\, :\, y^* \in J(y)\}\]
is called a semi-inner product.
\end{definition}

Consider the Banach space $X=C(M)$, where $M$ is a compact metric space and $C(M)$ denotes the space of all continuous real-valued 
functions on $M$ equipped with the maximum norm.
Let $\borel(M)$ denote space of all signed Borel measures on $M$, 
and let $\borel(M)^+$ denote space of all positive Borel measures on $M$. 

\begin{proposition}[Jordan decomposition] \label{Jordan:decomposition}
For any $\mu\in\borel(M)$, there exists a unique pair $(\mu_P,\mu_N)\in\borel^+(M)\times\borel^+(M)$
supported on Borel sets $P,N\subset M$ such that $\mu = \mu_P-\mu_N$ and $M$ is the disjoint union 
of $P$ and $N$.
\end{proposition}
For a proof, see Theorem 10 and Corollary 11 in Section III.4 of \cite{Dunford:Schwartz:59}.

As a consequence, the total variation of a signed Borel measure is well-defined.
\begin{definition}
The total variation of a Borel measure $\mu\in\borel(\R^d)$ with Jordan decomposition
$\mu_P+\mu_N=\mu$ with associated Borel sets $P\cup N=M$ is defined by
\[\vertiii{\mu}:=\mu_P(P)+\mu_N(N).\]
\end{definition}
It is well-known that the dual space of $(C(M),\|\cdot\|_\infty)$ is $(\borel(M),\vertiii{\cdot})$,
which follows from the Riesz representation theorem, see Theorem IV.6.3 in \cite{Dunford:Schwartz:59}.

We will now characterize the duality map on $C(M)$.
For a given function $f\in C(M)$, we define the sets
\begin{align*} 
E_f^P=\{x\in M: f(x)=\|f\|_\infty\},\quad E_f^N=\{x\in M: f(x)=-\|f\|_\infty\}
\end{align*}
on which $f$ attains its maximal modulus. 
Clearly, $E_f^P\cup E_f^N\neq\emptyset$.
Note that either $E_f^P\cap E_f^N=\emptyset$ or $E_f^P\cap E_f^N=M$, which happens if and only if $f\equiv 0$.

\begin{proposition} \label{dual:representation}
Let $M$ be a compact metric space, 
let $f\in C(M)$ and let $\mu\in\borel(M)$. 
Then $\mu\in J(f)$ if and only if 
\begin{align} \label{normalization}
\vertiii{\mu}=\|f\|_\infty
\end{align}
and the Jordan decomposition of $\mu$ satisfies
\begin{align} \label{localization}
\mu_P(M\setminus E_f^P)=0=\mu_N(M\setminus E_f^N).
\end{align}
\end{proposition}

\begin{proof}
Let $\mu\in J(f)$. Then, clearly, \eqref{normalization} holds.
Moreover, if 
\[\mu_P(M\setminus E_f^P)+\mu_N(M\setminus E_f^N)>0,\]
then
\begin{align*}
\mu(f) &= \int_Pfd\mu_P - \int_Nfd\mu_N\\
&< \big(\mu_P(E_f^P) + \mu_P(M\setminus E_f^P) + \mu_N(E_f^N) + \mu_N(M\setminus E_f^N)\big)\|f\|_\infty\\
&= \vertiii{\mu}\|f\|_\infty = \|f\|_\infty^2,
\end{align*}
which contradicts $\mu(f)=\|f\|_\infty^2$. Hence \eqref{localization} holds.

On the other hand, if \eqref{normalization} and \eqref{localization} hold, then
\begin{align*}
\mu(f) &= \int_{E_f^P}fd\mu_P - \int_{E_f^N}fd\mu_N\\
&= \left(\mu_P(E_f^P) + \mu_N(E_f^N)\right)\|f\|_\infty
= \vertiii{\mu}\|f\|_\infty = \|f\|_\infty^2,
\end{align*}
so that $\mu\in J(f)$.
\end{proof}

The following proposition provides an explicit formula for the semi-inner product on $C(M)$.
\begin{proposition} \label{semi:inner:explained}
Let $M$ be a compact metric space 
and let $f,g\in C(M)$. Then 
\begin{align*} 
\langle f,g\rangle_- &= \|g\|_\infty\min\{\min_{x\in E_g^P}f(x),\min_{x\in E_g^N}-f(x)\}
\end{align*}
with the convention $\min\emptyset=\infty$.
\end{proposition}
Note that $E_g^P=\emptyset=E_g^N$ is impossible, and that therefore the right-hand side 
is finite.
\begin{proof}
Since $g$ is continuous, the sets $E_g^P$ and $E_g^N$ are non\-emp\-ty and compact.
Since $f$ is continuous, it attains its minimum over $E_g^P$ at some $x_g^P\in E_g^P$ and its maximum
over $E_g^N$ at some $x_g^N\in E_g^N$.
As the Dirac measures $\delta_{x_g^P}$ and $\delta_{x_g^N}$ satisfy $\delta_{x_g^P}\in\borel(M)^+$ 
and $\delta_{x_g^N}\in\borel(M)^+$, and because of 
\[\delta_{x_g^P}(M\setminus E_g^P)=0=\delta_{x_g^N}(M\setminus E_g^N)\]
and $\|\delta_{x_g^P}\|=\|\delta_{x_g^N}\|=1$, Proposition \ref{dual:representation} implies 
$\|g\|_\infty\delta_{x_g^P}\in J(g)$ and $-\|g\|_\infty\delta_{x_g^N}\in J(g)$.
Therefore, Proposition \ref{dual:representation} yields
\begin{align*}
&\langle f,g\rangle_- = \inf\{\mu(f): \mu\in J(g)\}
\le \|g\|_\infty\min\{\delta_{x_g^P}(f),-\delta_{x_g^N}(f)\}\\
&= \|g\|_\infty\min\{f(x_g^P),-f(x_g^N)\}
= \|g\|_\infty\min\{\min_{x\in E_g^P}f(x),-\max_{x\in E_g^N}f(x)\}.
\end{align*}
It is easy to see that no $\mu\in J(g)$ yields a lower value.
\end{proof}

When $X=C(\S^{d-1})$ and $A,B\in\mc{K}_c(\R^d)$,
explicit expressions for the sets $E_{\sigma_A-\sigma_B}^P$ and $E_{\sigma_A-\sigma_B}^N$ 
can be obtained using the following proposition about variational inequalities.
\begin{proposition} \label{variational}
Let $A\in\mc{K}_c(\R^d)$, $a^*\in A$ and $x\in\R^d$. Then 
\begin{eqnarray}
\|x-a^*\|=\dist(x,A)\ &\Leftrightarrow&\ \langle x-a^*,a-a^*\rangle\le0\quad \text{for all}\ a\in A,\label{min:vip}\\
\|a^*-x\|=\dist(A,x)\ &\Leftrightarrow&\ \langle x-a^*,a-a^*\rangle\ge0\quad \text{for all}\ a\in A.\label{max:vip}
\end{eqnarray}
\end{proposition}
\begin{proof}
Inequality \eqref{min:vip} is standard (see e.g.\ \cite[Proposition 7.4]{Clarke:13}),
and \eqref{max:vip} can be obtained by an analogous proof.
\end{proof}

We are now in the position to characterize the sets $E_{\sigma_A-\sigma_B}^P$ and $E_{\sigma_A-\sigma_B}^N$.
\begin{proposition} \label{extremal:sets:explained}
Let $A,B\in\mc{K}_c(\R^d)$, and let $\sigma_A,\sigma_B\in\Sigma(\R^d)$ be the corresponding
support functions.
\begin{itemize}
\item [a)] If $A=B$, then $E_{\sigma_A-\sigma_B}^P=\S^{d-1}$.
\item [b)] If $A\subsetneq B$, then $E_{\sigma_A-\sigma_B}^P=\emptyset$.
\item [c)] Let $A\not\subset B$.
Then for any $p\in\S^{d-1}$, we have $p\in E_{\sigma_A-\sigma_B}^P$ if and only if
there exist $a^*\in A$ and $b^*\in B$ such that $p=(a^*-b^*)/\|a^*-b^*\|$ and
\begin{equation} \label{realise:distances}
\|a^*-b^*\|=\dist(a^*,B)=\dist(A,B)=\dist_H(A,B).  
\end{equation}
\end{itemize}
\end{proposition}

An analogous statement holds for the set $E_{\sigma_A-\sigma_B}^N$. 

\begin{proof}
If $A=B$, then $\sigma_A=\sigma_B$, and hence
\[E_{\sigma_A-\sigma_B}^P=\{p\in\S^{d-1}:\sigma_A(p)-\sigma_B(p)=\|\sigma_A-\sigma_B\|_\infty\}=\S^{d-1},\]
which proves a).
If $A\subsetneq B$, then $\|\sigma_A-\sigma_B\|_\infty>0$ and $\sigma_A-\sigma_B\le 0$,
so that
\[E_{\sigma_A-\sigma_B}^P=\{p\in\S^{d-1}:\sigma_A(p)-\sigma_B(p)=\|\sigma_A-\sigma_B\|_\infty\}=\emptyset,\]
which is b).

\medskip

Let us show the equivalence c). Let $p\in E_{\sigma_A-\sigma_B}^P$. 
Using Proposition \ref{elementary:support}, we find
\begin{align*}
\dist_H(A,B) &= \|\sigma_A-\sigma_B\|_\infty = \sigma_A(p)-\sigma_B(p)
= \sup_{a\in A}\langle p,a\rangle - \sup_{b\in B}\langle p,b\rangle\\
&= \sup_{a\in A}\inf_{b\in B}\langle p,a-b\rangle
= \sup_{a\in A}\inf_{b\in B}\cos\angle(p,a-b)\|a-b\|\\
&\le \sup_{a\in A}\cos\angle(p,a-\proj_B(a))\|a-\proj_B(a)\|.
\end{align*}
By compactness of $A$ and continuity of the above expression, there 
exists $a^*\in A$ such that 
\begin{align*}
\dist_H(A,B) &\le \sup_{a\in A}\cos\angle(p,a-\proj_B(a))\|a-\proj_B(a)\|\\ 
&= \cos\angle(p,a^*-\proj_B(a^*))\|a^*-\proj_B(a^*)\|\\
&= \cos\angle(p,a^*-\proj_B(a^*))\dist(a^*,B)\\
&\le \cos\angle(p,a^*-\proj_B(a^*))\dist(A,B) \le \dist(A,B). 
\end{align*}
Hence the above inequalities are, in fact, equalities, which enforces 
\begin{align*}
&\cos\angle(p,a^*-\proj_B(a^*))=1,\\
&0<\dist_H(A,B) = \dist(A,B) = \dist(a^*,B).
\end{align*}
Therefore, $a^*$ and $b^*:=\proj_B(a^*)\in B$ satisfy \eqref{realise:distances} and $p=(a^*-b^*)/\|a^*-b^*\|$.

\medskip

To show the opposite implication, let $a^*\in A$ and $b^*\in B$ satisfy \eqref{realise:distances}
and set $p=(a^*-b^*)/\|a^*-b^*\|$. 
Note that \eqref{realise:distances} and the assumption $A\not\subset B$ guarantee $a^*\neq b^*$.
Using \eqref{max:vip} and \eqref{min:vip}, we obtain 
\begin{align*}
\langle a^*-b^*,a\rangle \le \langle a^*-b^*,a^*\rangle\ \text{for all}\ a\in A,\\
\langle a^*-b^*,b\rangle \le \langle a^*-b^*,b^*\rangle\ \text{for all}\ b\in B,
\end{align*}
so that
\begin{align*}
\sup_{a\in A}\langle a^*-b^*,a\rangle = \langle a^*-b^*,a^*\rangle,\\
\sup_{b\in B}\langle a^*-b^*,b\rangle = \langle a^*-b^*,b^*\rangle.
\end{align*}
Hence, using Proposition \ref{elementary:support}, we find
\begin{align*}
\sigma_A(p)-\sigma_B(p) &= \sup_{a\in A}\langle p,a\rangle - \sup_{b\in B}\langle p,b\rangle\\ 
&= \tfrac{1}{\|a^*-b^*\|}\big(\sup_{a\in A}\langle a^*-b^*,a\rangle-\sup_{b\in B}\langle a^*-b^*,b\rangle \big)\\
&= \tfrac{1}{\|a^*-b^*\|}\big(\langle a^*-b^*,a^*\rangle-\langle a^*-b^*,b^*\rangle \big)\\
&= \|a^*-b^*\| = \dist_H(A,B) = \|\sigma_A-\sigma_B\|_\infty,
\end{align*}
so that $p\in E_{\sigma_A-\sigma_B}^P$.
\end{proof}

\section{Existence and uniqueness of solutions} \label{sec:ex:uni}

In this section we apply standard existence and uniqueness results for the initial value problem
\begin{equation} \label{general:ODE}
x'(t) = f(t,x(t)),\quad x(0)=x_0,
\end{equation}
on a real Banach space $X$ to the particular case of set differential equations in support function representation \eqref{support:ODE}. 
We first collect the necessary terminology and state a standard existence and uniqueness result for differential 
equations in Banach spaces from \cite{Deimling:77}. 

\begin{definition} \label{def:noncompactnessmeasures}
Let $X$ be a Banach space, and let $\mc{D}(X)$ be the family of all bounded subsets of $X$. 
The Kuratowski measure of non-compactness $\alpha:\mc{D}(X)\to\R$ is defined by
\begin{equation*}
\alpha(A) = \inf\{d>0: A\ \text{admits a finite covering by sets of diameter $\le d$}\}.
\end{equation*}
\end{definition}

Definition \ref{classes} introduces standard classes of growth functions from \cite{Deimling:77}.
The symbol $D^-$ denotes the Dini derivative
\[D^-\rho(t)=\liminf_{h\searrow 0}h^{-1}(\rho(t+h)-\rho(t))\]
of functions $\rho:\R\to\R$.
\begin{definition}  \label{classes}
We distinguish the following classes of growth functions.
\begin{itemize}
\item [(U0)] A continuous function $\omega:\R_0^+\rightarrow\R_0^+$ is said to be of class $U_0$ if the initial value problem
\[\rho'=\omega(\rho),\quad \rho(0)=0\]
possesses only the trivial solution.
\item [(U1)] Let $b>0$. A function $\omega:(0,b]\times\mathbb{R}_0^+\rightarrow\mathbb{R}$ is said to be of class $U_1$ if for each $\epsilon>0$ 
there is a $\delta>0$, a sequence $t_i\rightarrow 0^+$ and a sequence of continuous functions $\rho_i :[t_i, b]\rightarrow\mathbb{R}_0^+$ 
such that 
\begin{itemize}
\item [a)] $\rho_i(t_i)\ge \delta t_i$ for all $i\in\N$,
\item [b)] $0 < \rho_i(t)\le \epsilon$ for all $i\in\N$ and $t\in (t_i,b]$,
\item [c)] there exists a sequence $(\delta_i)_{i\in\N}$ with $\delta_i>0$ such that
$D^-\rho_i(t) \ge \omega(t,\rho_i(t))+\delta_i$ for all $i\in\N$ and $t\in (t_i,b]$.
\end{itemize}
\end{itemize}
\end{definition}

The following existence and uniqueness theorem is an excerpt of \cite[Theorem 4.1]{Deimling:77} applied in the 
present context. 
\begin{theorem} \label{ex:uni}
Let $(X,\|\cdot\|_X)$ be a Banach space, and let $D\subset X$, $x_0\in D$ and $r>0$ be such that $D_r:=D\cap B_r(x_0)$ is closed and convex.
Let $c>0$, let $f:[0,T]\times D_r\rightarrow X$ be a continuous function satisfying
\begin{equation*} \label{bound}
\|f(t,x)\|_X\le c\quad\text{for all}\ t\in[0,T],\ x\in D_r,
\end{equation*}
and let $b:=\min\{T,r/c\}$.
Suppose that the subtangent condition
\begin{equation*}
f(t,x) \in T_D(x)\quad\text{for all}\ t\in[0,b],\ x\in \partial D\cap B_r(x_0)
\end{equation*}
holds. Then the initial value problem \eqref{general:ODE} has a solution $\varphi:[0,b]\rightarrow D_r$ provided
one of the following additional conditions is satisfied:
\begin{itemize}
\item [a)] There exists a function $\omega:\R_0^+\rightarrow\R_0^+$ of class $U_0$ such that 
\[\alpha(f([0,b]\times A)) \le \omega(\alpha(A))\quad\text{for all}\ A\subset D_r.\]
\item [b)] There exists a function $\omega:(0,b]\times\R_+\rightarrow\R_+$ of class $U_1$ such that
\[\langle f(t,x)-f(t,y),x-y\rangle_- \le \omega(t,\|x-y\|_X)\|x-y\|_X\]
for all $t\in[0,b]$ and $x,y\in D_r$.
\end{itemize}
In case b), the solution is unique.
\end{theorem}

When adapting Theorem \ref{ex:uni} to set differential equations, we will frequently use the version 
\begin{equation} \label{subtangent}
f(t,\sigma)\in\overline{\Sigma(\R^d)-\R_0^+\sigma}\quad\text{for all}\ t\in[0,T],\ \sigma\in\Sigma(\R^d),
\end{equation}
of the subtangent condition to ensure that solutions do not leave the cone $\Sigma(\R^d)$
associated with $\mc{K}_c(\R^d)$.

Our first result is a Peano type theorem.
\begin{theorem} \label{Peano}
Let $f:[0,T]\times\Sigma(\R^d)\to\hat\Sigma(\R^d)$ be a continuous function, let $A_0\in\mc{K}_c(\R^d)$, and let $r>0$.
Then there exists $c>0$ such that 
\begin{equation} \label{compact:bound}
\|f(t,\sigma)\|_\infty\le c\quad \text{for all}\ t\in[0,T],\ \sigma\in D_r=\Sigma(\R^d)\cap B_r(\sigma_{A_0}). 
\end{equation}
Let $b:=\min\{T,r/c\}$. If, in addition, the subtangent condition \eqref{subtangent} holds,
then there exists a solution $\sigma:[0,b]\to D_r$ of the set differential equation \eqref{support:ODE}
in support function representation.
\end{theorem}
\begin{proof}
Since balls defined in the maximum norm are always convex and $\Sigma(\R^d)$ is a convex cone,
the intersection $D_r$ is convex.
By Proposition \ref{locally:compact}, the set $D_r$ is compact, and the existence of some $c>0$ such that 
\eqref{compact:bound} holds is implied by the continuity of $f$.
By Proposition \ref{tangent:to:cone}, condition \eqref{subtangent} implies 
\[f(t,\sigma) \in T_{\Sigma(\R^d)}(\sigma)\quad \text{for all}\ t\in[0,b],\ \sigma\in D_r.\]
By compactness of $D_r$ and continuity of $f$, the image $f([0,T]\times D_r)$ is compact, 
and hence we have
\[\alpha(f([0,T]\times A))=0=\alpha(A)\quad\text{for all}\ A\subset D_r,\]
so that the compactness assumptions of Theorem \ref{ex:uni}a) are trivially satisfied
with $\omega(\rho)=\rho$ of class $U_0$.
\end{proof}

The next result is a Picard-Lindel\"of type statement.
\begin{theorem} \label{picard}
Let $A_0\in\mc{K}_c(\R^d)$, and let $f:[0,T]\times\Sigma(\R^d)\to\hat\Sigma(\R^d)$ be continuous and Lipschitz continuous
in its second argument, i.e.\ we assume that there exists $L>0$ such that
\[\|f(t,\sigma_A)-f(t,\sigma_B)\|_\infty \le L\|\sigma_A-\sigma_B\|_\infty = L\dist_H(A,B)\]
for all $A,B\in\mc{K}_c(\R^d)$.
If, in addition, $f$ satisfies condition \eqref{subtangent}, then there exists a unique solution 
$\sigma:[0,T]\to\Sigma(\R^d)$ of \eqref{support:ODE}.
\end{theorem}
\begin{proof}
As $f$ is continuous and $[0,T]$ is compact, we have
\[\kappa:=\sup_{t\in[0,T]}\|f(t,\sigma_{A_0})\|_\infty<\infty,\]
and Lipschitz continuity of $f$ yields
\[c_r:=\sup_{t\in[0,T],\ \sigma\in B_r(\sigma_{A_0})\cap\Sigma(\R^d)}\|f(t,\sigma)\|_\infty \le Lr+\kappa.\]
Because of
\begin{align*}
&\langle f(t,\sigma)-f(t,\tilde\sigma),\sigma-\tilde\sigma\rangle_- 
= \inf_{\mu\in J(\sigma-\tilde\sigma)}\mu(f(t,\sigma)-f(t,\tilde\sigma))\\
&\le \inf_{\mu\in J(\sigma-\tilde\sigma)}\vertiii{\mu}\|f(t,\sigma)-f(t,\tilde\sigma)\|_\infty
\le L\|\sigma-\tilde\sigma\|_\infty^2
\end{align*}
for all $t\in[0,T]$ and $\sigma,\tilde\sigma\in\Sigma(\R^d)$,
and by the arguments in the preceding proof, all assumptions of Theorem \ref{ex:uni}b) are verified with $r=1$, $c=c_1$ and $\omega(t,s)=Ls$, 
so there exists a unique solution $\sigma_0(\cdot):[0,b_0]\rightarrow\Sigma(\R^d)\cap B_1(\sigma_{A_0})$ of \eqref{support:ODE} with
$b_0:=\min\{T,1/(L+\kappa)\}$.
If $1/(L+\kappa)<T$, the same argument yields a unique solution $\sigma_1(\cdot):[b_0,b_0+b_1]\rightarrow\Sigma(\R^d)\cap B_1(\sigma_0(b_0))$
of the set differential equation with $b_1:=\min\{T-b_0,1/(2L+\kappa)\}$.

Assume that $b_0+b_1<T$ and that this construction can be repeated indefinitely with $\sum_{k=0}^Nb_k<T$ for all $N\in\N$.
But then 
\[T\ge\sum_{k=0}^\infty b_k = \sum_{k=0}^\infty\tfrac{1}{kL+\kappa} = \infty,\]
which is a contradiction. 
Hence there exists a smallest index $N\in\N$ such that $b_N=T$.
Concatenating the unique solutions $\sigma_0,\ldots,\sigma_N$ yields a unique solution $\sigma:[0,T]\to\Sigma(\R^d)$
of \eqref{general:ODE} on the entire interval $[0,T]$.
\end{proof}

In contrast to the Picard-Lindel\"of type result above, the following statement fully exploits Theorem \ref{ex:uni}b)
and the considerations from Section \ref{sub:duality}.
Roughly speaking, it states that uniqueness of the solution can be guaranteed by controlling the relative velocity
$f(t,\sigma_A)-f(t,\sigma_B)$ for two sets $A,B\in\mc{K}_c(\R^d)$ in only one critical direction that is given by 
a pair $(a,b)\in A\times B$ which realizes the Hausdorff distance of $A$ and $B$.

\begin{figure}
\scriptsize
\begin{center}
\psfrag{A}{$A$}
\psfrag{B}{$B$}
\psfrag{a}{$a$}
\psfrag{b}{$b$}
\psfrag{v}{$v_r(p)$}
\psfrag{d}{$\dist_H(A,B)$}
\psfrag{f(A)}{$\tilde A$}
\psfrag{f(B)}{$\tilde B$}
\psfrag{f(A)(p)}{}
\psfrag{f(B)(p)}{}
\includegraphics[scale=0.5]{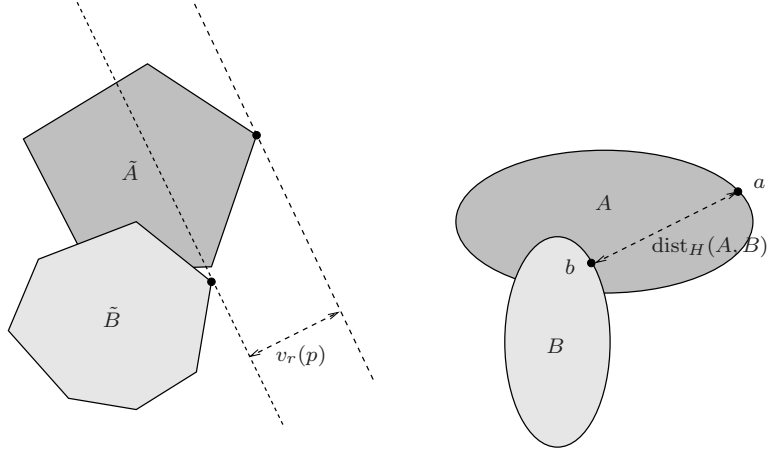}
\end{center}
\caption{Illustration of the geometric condition in Theorem \ref{OSL:thm} in an important special case.
Let $t\in(0,T)$ and assume that there exist $\tilde A,\tilde B\in\mc{K}_c(\R^d)$ such that $\sigma_{\tilde A}=f(t,\sigma_A)$
and $\sigma_{\tilde B}=f(t,\sigma_B)$.
Then the illustration depicts the relative velocity $v_r(p)=f(t,\sigma_A)(p)-f(t,\sigma_B)(p)$ in the critical direction 
$p=(a-b)/\|a-b\|$.\label{fig:vel}}
\end{figure}

\begin{theorem} \label{OSL:thm}
Let $f:[0,T]\times\Sigma(\R^d)\to\hat\Sigma(\R^d)$ be continuous, let $A_0\in\mc{K}_c(\R^d)$, and let $r>0$.
Let $b,c>0$ and $D_r$ be as in Theorem \ref{Peano}, let
\[D_r':=\{A\in\mc{K}_c(\R^d): \dist_H(A,A_0)\le r\},\]
let $\omega:(0,T]\times\mathbb{R}_+\rightarrow\mathbb{R}$ be of class $U_1$, and assume that the subtangent
condition \eqref{subtangent} holds.
If, in addiditon, for any $t\in[0,b]$ and $A,B\in D_r'$ with $A\neq B$, there exist
$a\in A$ and $b\in B$ such that $p:=(a-b)/\|a-b\|$ is well-defined and one of the conditions
\begin{align} \label{cond1}
\begin{split}
\|a-b\| = \dist(a,B) &= \dist(A,B) = \dist_H(A,B),\\
f(t,\sigma_A)(p)-f(t,\sigma_B)(p) &\le \omega(t,\dist_H(A,B))
\end{split}
\end{align}
and
\begin{align} \label{cond2}
\begin{split}
\|a-b\| = \dist(b,A) &= \dist(B,A) = \dist_H(A,B),\\
f(t,\sigma_B)(-p)-f(t,\sigma_A)(-p) &\le \omega(t,\dist_H(A,B))
\end{split}
\end{align}
is satisfied, then there exists a unique solution $\sigma:[0,b]\to D_r$ of \eqref{support:ODE}.
\end{theorem}

The geometric principle behind conditions \eqref{cond1} and \eqref{cond2} is depicted in Figure \ref{fig:vel}
for the case when $f(t,\sigma_A),f(t,\sigma_B)\in\Sigma$.

\begin{proof}
By Theorem \ref{Peano}, we know that the desired solution exists.
According to Theorem \ref{ex:uni}b), to ensure uniqueness, we need to verify that 
\[\langle f(t,\sigma_A)-f(t,\sigma_B),\sigma_A-\sigma_B\rangle_-\le\omega(t,\|\sigma_A-\sigma_B\|_\infty)\|\sigma_A-\sigma_B\|_\infty\]
for any $t\in(0,b]$ and $\sigma_A,\sigma_B\in D_r$. 
By Proposition \ref{semi:inner:explained} this is true if and only if for any $t\in(0,b]$ and $\sigma_A,\sigma_B\in D_r$, 
at least one of the inequalities
\begin{align*}
\min_{p\in E_{\sigma_A-\sigma_B}^P}(f(t,\sigma_A)(p)-f(t,\sigma_B)(p))&\le\omega(t,\dist_H(A,B)),\\
\min_{p\in E_{\sigma_B-\sigma_A}^P}(f(t,\sigma_B)(p)-f(t,\sigma_A)(p))&\le\omega(t,\dist_H(A,B))
\end{align*}
is satisfied. 
If $\sigma_A\neq\sigma_B$, this is, according to Proposition \ref{extremal:sets:explained}, ensured by conditions
\eqref{cond1} and \eqref{cond2}, which can be checked by addressing all possible relations $A\subsetneq B$,
$B\subsetneq A$ and $A\not\subset B \wedge B\not\subset A$ between the sets $A$ and $B$ separately.
If $\sigma_A=\sigma_B$, both inequalities are obviously valid.
\end{proof}

\section{Hukuhara-type differentials} \label{sec:second}

In this section, we clarify that curves $A:[0,T]\to\mc{K}_c(\R^d)$, which are 
second type Hukuhara differentiable, are time-reversed Hukuhara differentiable curves
with the same derivative up to sign change.
This insight has some important consequences.
\begin{itemize}
\item [i)] As Hukuhara differentiable curves can only grow
in diameter, see \cite[Proposition 1.6.1]{Lakshmikantham}, second type Hukuhara 
differentiable curves can only shrink in diameter, as claimed in the introduction.
\item [ii)] As the support function representation of Hukuhara differentiable curves is Frechet
differentiable, see \cite[Lemma 4.1]{Banks:Jacobs:70}, this also holds for second type Hukuhara differentiable curves. 
Furthermore, by the same lemma, the Hukuhara and the second type Hukuhara differentials of a curve coincide with its 
Frechet differential (up to a sign change), whenever the Hukuhara type differentials exist.
Therefore, set differential equations based on both types of Hukuhara derivatives are
special cases of the support function approach we presented.
\end{itemize}

The notions of Hukuhara difference and Hukuhara differential are standard.
The concept of generalized or second type Hukuhara differentials goes back to \cite{Bede:Gal:04}.
Their use for set differential equations was investigated in \cite{Malinowski:12a,Malinowski:12}.

\begin{definition} (Hukuhara differences and differentials)
\begin{itemize}
\item [a)] Let $A,B\in\mc{K}_c(\R^d)$. If there exists $C\in\mc{K}_c(\R^d)$ such that
$A=B+C$, then $C$ is called the Hukuhara difference between $A$ and $B$,
and we denote $C=A\ominus_H B$.
\item [b)] A curve $A:[0,T]\to\mc{K}_c(\R^d)$ is called Hukuhara differentiable
at $t\in(0,T)$ with Hukuhara differential $D_HA(t)\in\mc{K}_c(\R^d)$ if the limits
\[\lim_{h\searrow 0}h^{-1}\big(A(t+h)\ominus_H A(t)\big),\quad \lim_{h\searrow 0}h^{-1}\big(A(t)\ominus_H A(t-h)\big)\]
w.r.t.\ Hausdorff distance exist and equal $D_HA(t)$.
\item [c)] A curve $A:[0,T]\to\mc{K}_c(\R^d)$ is called second type Hukuhara differentiable
at $t\in(0,T)$ with differential $D_H^*A(t)\in\mc{K}_c(\R^d)$ if the limits
\[\lim_{h\searrow 0}(-h)^{-1}\big(A(t)\ominus_H A(t+h)\big),\quad \lim_{h\searrow 0}(-h)^{-1}\big(A(t-h)\ominus_H A(t)\big)\]
w.r.t.\ Hausdorff distance exist and equal $D_H^*A(t)$.
\end{itemize}
\end{definition}

The following proposition shows that second type Hukuhara differentiable curves 
are precisely those curves that are Hukuhara differentiable in the ordinary sense
after time reversal.

\begin{proposition} \label{Hukuhara:backwards}
Let $A:[0,T]\to\mc{K}_c(\R^d)$ be a curve, and let $B:[-T,0]\to\mc{K}_c(\R^d)$ be given by $B(t)=A(-t)$.
Then $A$ is second type Hukuhara differentiable at $t\in(0,T)$ if and only if $B$ is Hukuhara differentiable at $-t$ in the usual sense.
In that case, the respective differentials satisfy 
\[D_H^*A(t)=-D_HB(-t).\]
\end{proposition}
\begin{proof}
The statement follows immediately from the identities
\begin{align*}
\lim_{h\searrow 0}(-h)^{-1}\big(A(t)\ominus_H A(t+h)\big) &= -\lim_{h\searrow 0}h^{-1}\big(B(-t)\ominus_H(B(-t-h))\big),\\
\lim_{h\searrow 0}(-h)^{-1}\big(A(t-h)\ominus_H A(t)\big) &= -\lim_{h\searrow 0}h^{-1}\big(B(-t+h)\ominus_H B(-t)\big)
\end{align*}
for the Hausdorff limits.
\end{proof}

\begin{figure}
\begin{center}
\includegraphics[scale=0.75]{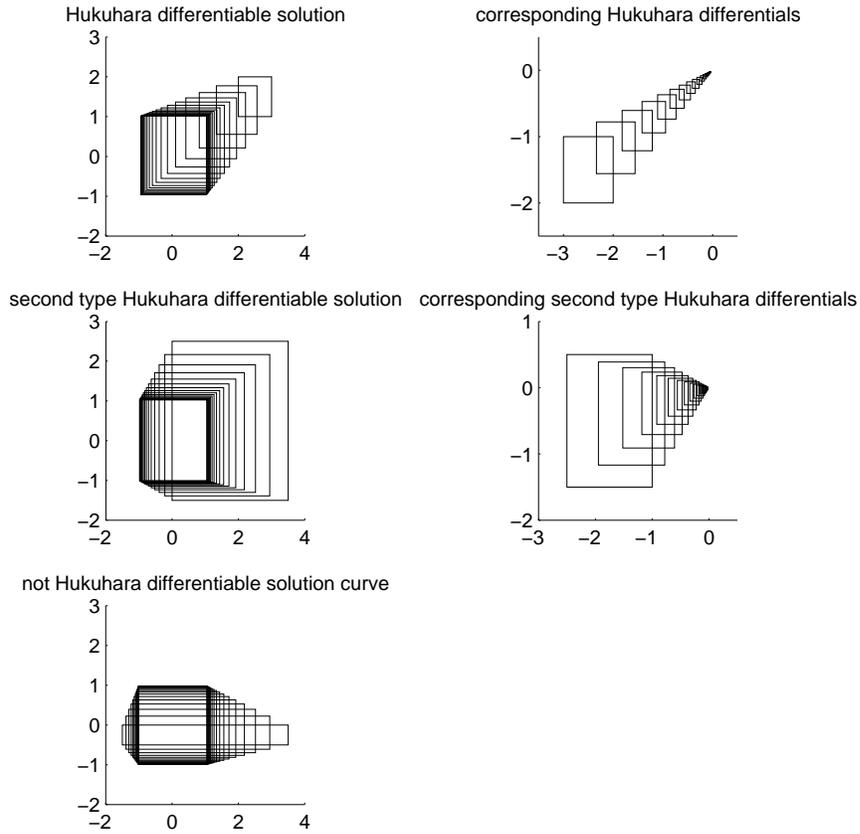}
\end{center}
\caption{Solutions to set differential equation \eqref{example:ode} with three different initial values. 
The rectangles in the frames on the left are the values $A(t)$, $t=0,\tfrac14,\tfrac12,\tfrac34,\ldots$, of the solutions.
The rectangles in the top right frame are the Hukuhara differentials $D_HA(t)$, and the rectangles in the second frame on the 
right are the second type Hukuhara differentials $D_H^*A(t)$ at the same time points.
The bottom right frame is empty, because the third solution curve is neither Hukuhara nor second type Hukuhara differentible.
\label{fig:set}}
\end{figure}

\begin{figure}
\begin{center}
\includegraphics[scale=0.75]{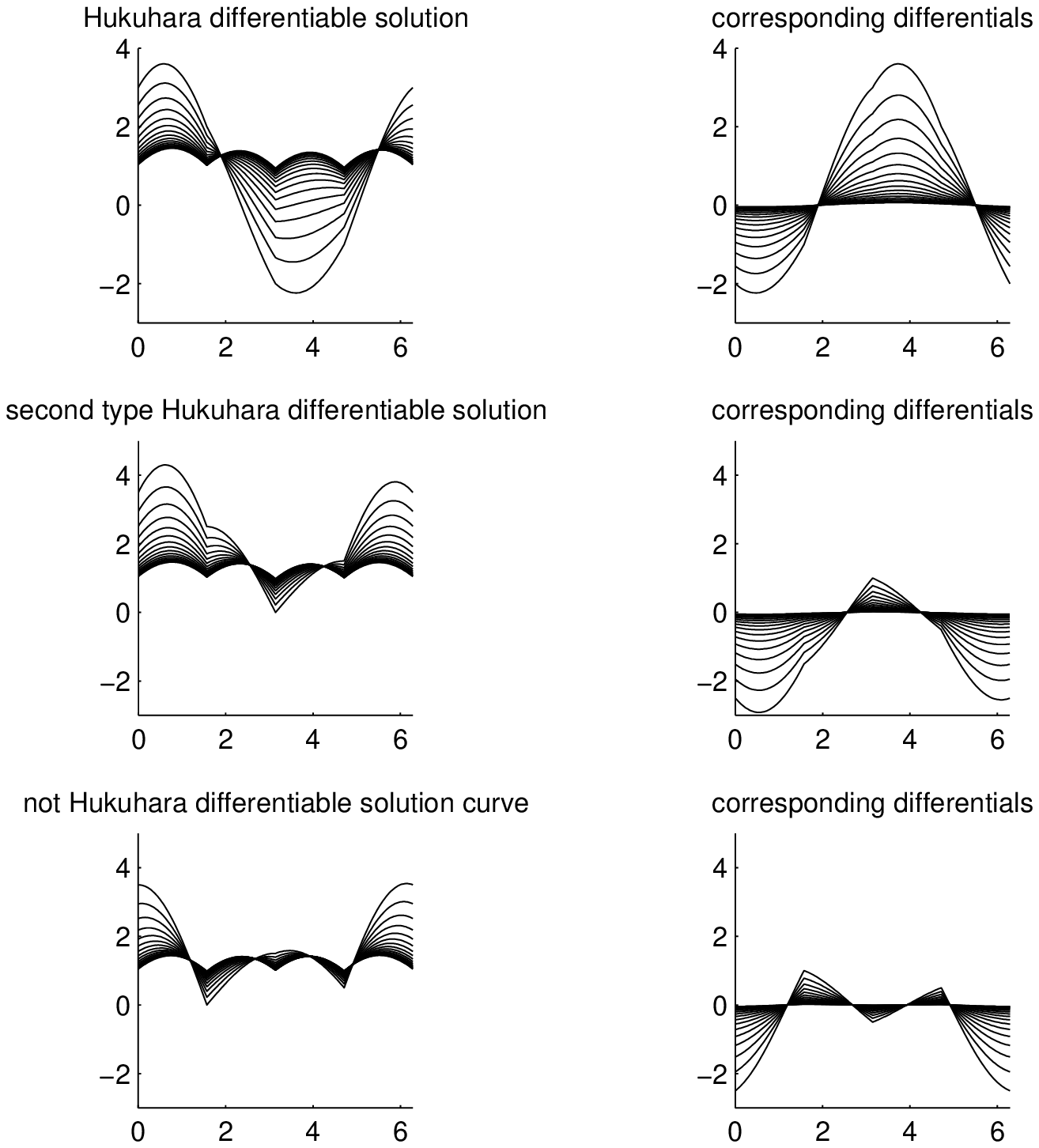}
\end{center}
\caption{Solutions from Figure \ref{fig:set} and the corresponding differentials 
in support function representation.
The differentials of the third curve cannot be interpreted as sets, but are well-defined as elements of $\hat\Sigma(\R^d)$.\label{fig:sup}}
\end{figure}

\section{Example} \label{counterexample}
We conclude our paper with a simple, but instructive example, which illustrates that the usefulness 
of both types of Hukuhara derivative depends not only on the equation, but even on the initial value.
Consider the set differential equation
\begin{equation} \label{example:ode}
\tfrac{d}{dt}\sigma_{A(t)}=\sigma_Q-\sigma_{A(t)},\quad \sigma_{A(0)}=\sigma_{A_0}
\end{equation}
in $\mc{K}_c(\R^2)$ with $Q=[-1,1]^2$ and $A_0=[a_1,b_1]\times[a_2,b_2]\subset\R^2$.
The curve
\begin{equation} \label{sol}
A(t)=e^{-t}A_0+(1-e^{-t})Q
\end{equation}
is a solution of \eqref{example:ode}, because
\[\tfrac{d}{dt}\sigma_{A(t)}=e^{-t}(\sigma_{Q}-\sigma_{A_0})=\sigma_Q-\sigma_{A(t)}.\]
By Theorem \ref{picard}, the solution is unique.
Clearly, the set $Q$ is a globally asymptotically stable fixed point.

By \cite[Lemma 4.1]{Banks:Jacobs:70}, any Hukuhara differentiable solution of the reformulation
\begin{equation} \label{Hukuhara:equation}
D_HA(t)=Q\ominus_H A(t),\quad A(0)=A_0 
\end{equation}
of \eqref{example:ode} in set notation must coincide with this curve.
Note that for many $A\in\mc{K}_c(\R^2)$, the right-hand side $Q\ominus_H A$ of \eqref{Hukuhara:equation} is not well-defined.
Since $\sigma_{Q}-\sigma_{A_0}\in\Sigma(\R^d)$ if and only if 
\begin{equation} \label{cond:ex}
\max\{b_1-a_1,b_2-a_2\}\le 2,
\end{equation}
there does not exist a Hukuhara differentiable solution if this condition is violated.
A computation shows that \eqref{cond:ex} is sufficient for \eqref{sol} being a solution
of first Hukuhara type.

Proposition \ref{Hukuhara:backwards}, however, shows that the curve \eqref{sol}
can only be a second type Hukuhara solution, if $\sigma_{Q}-\sigma_{A_0}\in-\Sigma(\R^d)$,
which is equivalent with 
\begin{equation} \label{cond:ex:2}
\min\{b_1-a_1,b_2-a_2\}\ge 2,
\end{equation}
and condition \eqref{cond:ex:2}
is sufficient for \eqref{sol} being a solution of second Hukuhara type.

\medskip

Figures \ref{fig:set} and \ref{fig:sup} display solutions of \eqref{example:ode} with three different
initial values $A_0^1=[2,3]\times[1,2]$, $A_0^2=[0,3.5]\times[-1.5,2.5]$ and $A_0^3=[-1.5,3.5]\times[-0.5,0]$.

Figure \ref{fig:set} depicts the sets as such on the left. 
It is clearly visible that $\dist_H(A(t),Q)\to 0$ as $t\rightarrow\infty$.
The first curve is Hukuhara, but not second type Hukuhara differentiable, and the Hukuhara differentials 
are plotted in the top right subplot.
The second curve is second type Hukuhara, but not Hukuhara differentiable, and the second type Hukuhara
differentials are plotted in the middle of the right column.
In both cases, the differentials converge to $\{0\}$ when the state approaches $Q$.
The third curve is neither Hukuhara nor second type Hukuhara differentiable, because it shrinks in 
the direction of the first and grows in the direction of the second axis.

Figure \ref{fig:sup} depicts the same three curves in support function representation.
The left column shows the evolution of the support functions, while the right column shows 
the Frechet differentials along that curve.
In this representation, the third curve can be treated as any other.
The fact, that its differentials are elements of $\hat\Sigma(\R^d)\setminus\Sigma(\R^d)$
causes no problems.
In all three cases, the derivatives converge to the zero function as the state approaches the equilibrium.

\medskip

We conclude that both types of Hukuhara differentiability only yield solutions for very 
special initial conditions, while the support function approach yields a solution that exhibits 
the expected behavior for any initial condition without any technical complications.

\bibliographystyle{plain}
\bibliography{sode}

\end{document}